\documentclass[11pt]{amsart}
\usepackage{amsthm}
\usepackage{cite}
\usepackage{array}
\usepackage{amsmath}
\usepackage{enumerate}
\usepackage{tikz}
\usetikzlibrary{calc}
\usepackage{amssymb}
\usepackage{latexsym}
\usepackage{amsfonts}
\usepackage{color}
\usepackage{mathrsfs}
\usepackage{epsfig,helvet}

\newcommand{\Z}{{\mathbb Z}}

\theoremstyle{plain} \numberwithin{equation}{section}
\newtheorem{thm}{Theorem}[section]
\newtheorem{theorem}[thm]{Theorem}
\newtheorem{lemma}[thm]{Lemma}

\newtheorem{definition}[thm]{Definition}

\begin{document}
\setcounter{page}{1}

\title[ Hasan and Padhan]{On the triple tensor product of generalized Heisenberg Lie superalgebra of rank $\leq2$}

\author{IBRAHEM YAKZAN HASAN}
\address{Centre for Data Science, Institute of Technical Education and Research \\
	Siksha `O' Anusandhan (A Deemed to be University)\\
	Bhubaneswar-751030 \\
	Odisha, India}
\email{ibrahemhasan898@gmail.com}
\author[Padhan]{Rudra Narayan Padhan}
\address{Centre for Data Science, Institute of Technical Education and Research  \\
	Siksha `O' Anusandhan (A Deemed to be University)\\
	Bhubaneswar-751030 \\
	Odisha, India}
\email{rudra.padhan6@gmail.com, rudranarayanpadhan@soa.ac.in}

\maketitle

\begin{abstract}
In this article, we compute the Schur multiplier of all generalized Heisenberg Lie superalgebras of rank $2$. We discuss the structure of  $\otimes^3H$ and $\wedge^3H$ where $H$ is a generalized Heisenberg Lie superalgebra of rank $\leq2$. Moreover, we prove that if $L$ is an $(m\mid n)$-dimensional non-abelian nilpotent Lie superalgebra with derived subalgebra of dimension $(r\mid s)$, then $\dim\otimes^3L \leq (m+n)(m+n - (r+s))^2$. In particular, for $r=1,s=0$ the equality holds if and only if $L \cong H(1\mid 0)$.
\end{abstract}

\section{Introduction}

Lie superalgebras have applications in many areas of Mathematics and Theoretical Physics as they can be used to describe supersymmetry. Kac \cite{Kac1977} gives a comprehensive description of the mathematical theory of Lie superalgebras, and establishes the classification of all finite dimensional simple Lie superalgebras over an algebraically closed field of characteristic zero. In the last few years, the theory of Lie superalgebras has evolved remarkably, obtaining many results in representation theory and classification. Most of the results are extensions of well known facts of Lie algebras. But the classification of all finite dimensional nilpotent Lie superalgebras is still an open problem like that of finite dimensional nilpotent Lie algebras. In the last few years, the theory of Lie superalgebras has evolved remarkably, obtaining many results in representation theory and classification. Most of the results are an extension of well-known facts of Lie algebras \cite{MC2011, Musson2012, GKL2015}.  Recently, Garc\'{i}a-Mart\'{i}nez \cite{GKL2015} introduced the notions of non-abelian tensor product of Lie superalgebras and the exterior product of Lie superalgebras over a commutative ring. In this paper, we determine an upper bound on the dimension of the non-abelian tensor product of nilpotent Lie superalgebra.

\smallskip

In 1991, Rocco \cite{Rocco} proved that if $G$ is a finite $p$-group of order $P^n$ with the derived subgroup of order $p^m$, then $G\otimes G \leq p^{n(n-m)}$. Furthermore, this bound was improved by Niroomand \cite{p2010}, i.e., $G\otimes G \leq p^{n(n-m)}$. Ellis \cite{Ellis1987,Ellis1991} developed the theory of tensor product for Lie algebra. If $L$ and $K$ are two finite dimensional nilpotent Lie algebras, then an upper bound and lower bound of the dimension of the non-abelian tensor product $L\otimes K$ have been studied by salemkar et al. \cite{s2010} and the results are a generalization of Rocco's results for finite $p$-group.  Recently, an improved upper bound on the dimension of  $L\otimes L$ has been discussed in \cite{Niroomand}, i.e., if $L$ is a non-abelian nilpotent Lie algebra of dimension $n$ and its derived subalgebra has dimension $m$, then $ \dim (L\otimes L) \leq (n-m)(n-1)+2$. 
Recently, Padhan et al. in \cite{hasanpadprad} have determined an upper bound for the non-abelian tensor product of finite dimensional Lie superalgebra. More precisely, they proved that if $L$ is a non-abelian nilpotent Lie superalgebra of dimension $(k \mid l)$ and its derived subalgebra has dimension $(r \mid s)$, then $ \dim (L\otimes L) \leq (k+l-(r+s))(k+l-1)+2$. They discussed the conditions when equality holds for $r=1, s=0$ explicitly.

\smallskip

A Lie superalgebras is said to be generalized Heisenberg Lie superalgebras of rank $(m\mid n)$ if $L'=Z(L)$ and $\dim Z(L)=(m \mid n)$. In this paper, we intend to compute the Schur multiplier of all generalized Heisenberg Lie superalgebras of rank $2$. Here we give the structure of the triple non-abelian tensor product and the exterior product of the generalized Heisenberg Lie superalgebra of rank $\leq2$. In the last section, we prove that if $L$ is a $(m\mid n)$-dimensional non-abelian nilpotent Lie superalgebra with derived subalgebra of dimension $(r\mid s)$, then $\dim\otimes^3L \leq (m+n)(m+n - (r+s))^2$ and we show that if $L \cong H(1\mid 0)$ then the equality holds.

\section{Preliminaries}

Let $\mathbb{Z}_{2}=\{\bar{0}, \bar{1}\}$ be a field. A $\mathbb{Z}_{2}$-graded vector space $V$ is simply a direct sum of vector spaces $V_{\bar{0}}$ and $V_{\bar{1}}$, i.e., $V = V_{\bar{0}} \oplus V_{\bar{1}}$. It is also referred as a superspace. We consider all vector superspaces and superalgebras are over field $\mathbb{F}$ (characteristic of $\mathbb{F} \neq 2,3$). Elements in $V_{\bar{0}}$ (resp. $V_{\bar{1}}$) are called even (resp. odd) elements. Non-zero elements of $V_{\bar{0}} \cup V_{\bar{1}}$ are called homogeneous elements. For a homogeneous element $v \in V_{\sigma}$, with $\sigma \in \mathbb{Z}_{2}$ we set $|v| = \sigma$ as the degree of $v$. A  subsuperspace (or, subspace)  $U$ of $V$ is a $\mathbb{Z}_2$-graded vector subspace where  $U= (V_{\bar{0}} \cap U) \oplus (V_{\bar{1}} \cap U)$. We adopt the convention that whenever the degree function appears in a formula, the corresponding elements are supposed to be homogeneous. 

\smallskip 

A  Lie superalgebra (see \cite{Kac1977, Musson2012}) is a superspace $L = L_{\bar{0}} \oplus L_{\bar{1}}$ with a bilinear mapping $ [., .] : L \times L \rightarrow L$ satisfying the following identities:
\begin{enumerate}
\item $[L_{\alpha}, L_{\beta}] \subset L_{\alpha+\beta}$, for $\alpha, \beta \in \mathbb{Z}_{2}$ ($\mathbb{Z}_{2}$-grading),
\item $[x, y] = -(-1)^{|x||y|} [y, x]$ (graded skew-symmetry),
\item $(-1)^{|x||z|} [x,[y, z]] + (-1)^{ |y| |x|} [y, [z, x]] + (-1)^{|z| |y|}[z,[ x, y]] = 0$ (graded Jacobi identity),
\end{enumerate}
for all $x, y, z \in L$. Clearly $L_{\bar{0}}$ is a Lie algebra, and $L_{\bar{1}}$ is a $L_{\bar{0}}$-module. If $L_{\bar{1}} = 0$, then $L$ is just Lie algebra, but in general, a Lie superalgebra is not a Lie algebra.  A Lie superalgebra $L$ is called abelian if  $[x, y] = 0$ for all $x, y \in L$. Lie superalgebras without even part, i.e., $L_{\bar{0}} = 0$, are  abelian. A subsuperalgebra (or subalgebra) of $L$ is a $\mathbb{Z}_{2}$-graded vector subspace that is closed under bracket operation. The graded subalgebra $[L, L]$, of $L$, is known as the derived subalgebra of $L$. A $\mathbb{Z}_{2}$-graded subspace $I$ is a graded ideal of $L$ if $[I, L]\subseteq I$. The ideal 
\[Z(L) = \{z\in L : [z, x] = 0\;\mbox{for all}\;x\in L\}\] 
is a graded ideal and it is called the {\it center} of $L$. A homomorphism between superspaces $f: V \rightarrow W $ of degree $|f|\in \mathbb{Z}_{2}$, is a linear map satisfying $f(V_{\alpha})\subseteq W_{\alpha+|f|}$ for $\alpha \in \mathbb{Z}_{2}$. In particular, if $|f| = \bar{0}$, then the homomorphism $f$ is called the homogeneous linear map of even degree. A Lie superalgebra homomorphism $f: L \rightarrow M$ is a  homogeneous linear map of even degree such that $f([x,y]) = [f(x), f(y)]$ holds for all $x, y \in L$.  If $I$ is an ideal of $L$, the quotient Lie superalgebra $L/I$ inherits a canonical Lie superalgebra structure such that the natural projection map becomes a homomorphism. The notions of {\it epimorphisms, isomorphisms,} and {\it automorphisms} have obvious meaning. 
\smallskip

Throughout this article, for the superdimension of Lie superalgebra $L$, we simply write $\dim L=(m\mid n)$, where $\dim L_{\bar{0}} = m$ and $\dim L_{\bar{1}} = n$. Also,  $A(m \mid n)$ denotes an abelian Lie superalgebra where $\dim A=(m\mid n)$. A  Lie superalgebra $L$ is said to be a Heisenberg Lie superalgebra if $Z(L)=L'$ and $\dim Z(L)=1$. According to the homogeneous generator of $Z(L)$, Heisenberg Lie superalgebras can further split into even or odd Heisenberg Lie superalgebras \cite{MC2011}. By Heisenberg Lie superalgebra we mean special Heisenberg Lie superalgebra in this article. For more details on Heisenberg Lie superalgebra and its,  multiplier see \cite{Nayak2018,SN2018b,Padhandetec,p50, p54,p55,hesam}. Also, for more details on the Schur multiplier and capability of Lie algebras, see \cite{Alamian2008,Batten1993,Batten1996,BS1996,p1,Ellis1987,Ellis1991,Ellis1995,Ellis1996,
Hardy1998,Hardy2005,org,Kar1987,Moneyhun1994,N2,Niroomand2011,Russo2011,PMF2013}. Now we list some useful results from \cite{Nayak2018}, for further use.

\smallskip

The descending central sequence of a Lie superalgebra $L = L_{\bar{0}} \oplus L_{\bar{1}}$ is defined by $L^{1} = L, L^{k+1} = [L^{k}, L]$ for all $k\geq 1$. If $L^{k+1} = \{0\}$ for some $k+1$, with $L^{k} \neq 0$, then $L$ is called nilpotent with nilpotency class $k$. One can observe that generalized Heisenberg Lie superalgebra of rank $(r|s)$ is of nilpotency class two. We denote $\gamma_k(L)$ be the $k$-th term of the descending central series of $L$.

\smallskip

The concepts of stem extension, cover and multiplier of Lie superalgebra are defined and studied  \cite{Nayak2018, SN2018a, NPP}. One can see \cite{p51,p52,p53} for more details on isoclinism and the central extension of Lie superalgebras.  Free presentation of a Lie superalgebra $L$ is the extension $0\longrightarrow R \longrightarrow F\longrightarrow L$, where $F$ is a free Lie superalgebra, then $\mathcal{M}( L) \cong F'\cap R / [F,R]$.
 
\begin{lemma}\label{th3.3}\cite[See Theorem 3.4]{Nayak2018}
\[\dim \mathcal{M}(A(m \mid n)) = \big(\frac{1}{2}(m^2+n^2+n-m)\mid mn \big).\]
\end{lemma}

\smallskip 

A  Lie superalgebra $L$ is said to be Heisenberg Lie superalgebra if $Z(L)=L'$ and $\dim Z(L)=1$. According to the homogeneous generator of $Z(L)$, Heisenberg Lie superalgebras can further split into even or odd Heisenberg Lie superalgebras \cite{MC2011}.

\begin{lemma} \label{th3.4}\cite[See Theorem 4.2, 4.3]{Nayak2018}
Every Heisenberg Lie superalgebra with an even center has dimension $(2m+1 \mid n)$, is isomorphic to $H(m , n)=H_{\overline{0}}\oplus H_{\overline{1}}$, where
\[H_{\overline{0}}=<x_{1},\ldots,x_{2m},z \mid [x_{i},x_{m+i}]=z,\ i=1,\ldots,m>\]
and 
\[H_{\overline{1}}=<y_{1},\ldots,y_{n}\mid [y_{j}, y_{j}]=z,\  j=1,\ldots,n>.\]
Further,
$$ 
\dim \mathcal{M}(H(m , n))=
\begin{cases}
 (2m^{2}-m+n(n+1)/2-1 \mid 2mn)\quad \mbox{if}\;m+n\geq 2\\
(0 \mid 0) \quad \mbox{if}\;m=0, n=1\\  
   (2 \mid 0)\quad \mbox{if}\;m=1, n=0.
\end{cases}
$$

\end{lemma}

Similarly the multiplier and cover for Heisenberg Lie superalgebra of odd center is known.

\begin{lemma}\label{th3.6}\cite[See Theorem 2.8]{SN2018b}
Every Heisenberg Lie superalgebra, with an odd center has dimension $(m \mid m+1)$, is isomorphic to $H_{m}=H_{\overline{0}}\oplus H_{\overline{1}}$, where
\[H_{m}=<x_{1},\ldots,x_{m} , y_{1},\ldots,y_{m},z \mid [x_{j},y_{j}]=z,   j=1,\ldots,m>.\]
Further,
$$
\dim \mathcal{M}(H_{m})=
\begin{cases}
 (m^{2}\mid m^{2}-1)\quad \mbox{if}\;m\geq 2\\  
(1\mid 1) \quad \quad \mbox{if}\;m=1.  
  \end{cases}
$$
\end{lemma}

For any two Lie superalgebras $H$ and $K$ the Lie superalgebra direct sum $H \oplus K$ is a Lie superalgebra with natural grading $(H \oplus K)_{\alpha}=H_{\alpha}\oplus K_{\alpha}$ where $\alpha \in \mathbb{Z}_2$. If $\mathcal{M}(H)$ and $\mathcal{M}(K)$ are known, then $\mathcal{M}(H \oplus K)$ is given by the following result.
\begin{lemma}\label{th3.7}\cite[See Theorem 3.9]{Nayak2018}
For Lie superalgebras $H$ and $K$,
\[\mathcal{M}(H\oplus K)\cong \mathcal{M}(H)\oplus \mathcal{M}(K)\oplus (H/H^2\otimes K/K^2).\]
\end{lemma}

\begin{lemma}\label{th4.4}\cite[Proposition 3.4]{SN2018b}\cite[Theorem 6.9]{Padhandetec} \label{th5a}
Let $L$ be a nilpotent Lie superalgebra of dimension $(k \mid l)$ with $\dim L'=(r \mid s)$, where $r+s=1$. If $r=1, s=0$ then $L \cong H(m,n)\oplus A(k-2m-1 \mid l-n)$ for $m+n\geq 1$. If $r=0, s=1$ then $L \cong H_{m} \oplus A(k-m \mid l-m-1)$. Moreover, $L$ is capable if and only if either $L \cong H(1 , 0)\oplus A(k-3 \mid l)$ or $L \cong H_{1}\oplus A(k-1 \mid l-2)$.
\end{lemma}

Consider the exact sequence of Lie superalgebras (see \cite{hesam})
\begin{align*}
&\mathcal{H}_3(L) \longrightarrow \mathcal{H}_3(L/I) \longrightarrow \mathcal{M}(L, I) \longrightarrow \mathcal{M}(L) \longrightarrow \mathcal{M}(L/I)\\ 
& \longrightarrow I/[L,I] \longrightarrow L/L^2 \longrightarrow L/[L^2+I] \longrightarrow 0.
\end{align*}
Then the schur multiplier of pair of Lie superalgebras $(L,I)$  is the abelian Lie superalgebra $\mathcal{M}( L,I)$ which appear in the above sequence. This definition is analogous to the definition of the Schur multiplier of a pair of groups and pair of Lie algebras \cite{Ellis1996, org}. Free presentation of a Lie superalgebra $L$ is the extension $0\longrightarrow R \longrightarrow F\longrightarrow L$, where $F$ is a free Lie superalgebra. If $I$ is a graded ideal of $L$ such that $I\cong S/R$ for some graded ideal $S$ of $F$, then $$\mathcal{M}( L,I) \cong R\cap [F,S] / [F,R].$$

\begin{lemma} \label{c001}\cite[Corollary 4.6]{p55}
Let $I$ be a graded ideal of an abelian Lie superalgebra $A(m \mid n)$. Then 
	\[\mathcal{M}(A(m \mid n),I)\cong  A(m \mid n)\wedge I\]
\end{lemma}

\begin{lemma}\label{c002}\cite[Corollary 4.7]{p55}
Let $I$ be a $(k\mid h)$-dimensional ideal of $(A(m \mid n)$. Then 
	\[\dim \mathcal{M}(A(m \mid n),I)=(\frac{1}{2}[k(2m-k-1)+h(2n-h+1)] \mid mn-(m-k)(n-h)).\] 
\end{lemma}

\section{ Non-abelian tensor products  and some known results of Lie superalgebras}

Consider $\mathbb{K}$ is a commutative ring. Recently, Garc\'{i}a-Mart\'{i}nez \cite{GKL2015} introduced the notions of non-abelian tensor product of Lie superalgebras and exterior product of Lie superalgebras over $\mathbb{K}$. Here we recall some of the known notation and results from \cite{GKL2015}.

Let $P$ and $M$ be two Lie superalgebras, then by an action of $P$ on $M$ we mean a $\mathbb{K}$-bilinear map of even degree, \[P\times M \longrightarrow M,~~~~~~~(p, m)\mapsto {}^ pm, \]
such that 
\begin{enumerate}
	\item ${}^{[p , p']} m = {}^p({}^{p'}m)-(-1)^{|p||p'|}~  {}^{p'}({}^{p}m),$
	
	\item ${}^p{[m , m']}=[{}^pm, m']+(-1)^{|p||m|}[m , {}^pm'],$
	
\end{enumerate}
for all homogeneous $p, p' \in P$ and $m, m' \in M$. For any Lie superalgebra $M$, the Lie multiplication induces an action on itself via ${}^mm'=[m , m']$.  The action of $P$ on $M$ is called trivial if ${}^pm=0$ for all $p \in P$ and $m \in M$. \\

Given two Lie superalgebras $M$ and $P$ with an action of $P$ on $M$, we define the semidirect product $M \rtimes P$ with underlying supermodule $M \oplus P$ endowed with the bracket given by
$[(m, p), (m', p')]=([m, m']+ {}^pm'-(-1)^{|m||p'|}({}^{p'}m), [p, p'])$. A crossed module of Lie superalgebras is a homomorphism of Lie superalgebras $\partial : M\longrightarrow P$ with an action of $P$ on $M$ satisfying 
\begin{enumerate}
	\item $\partial ({}^p{m})=[p,\partial(m)],$
	
	\item ${}^{\partial(m)}{m'}=[m, m'],$ for all $p \in P$ and $m, m' \in M$.
	
\end{enumerate} 

A bilinear function $f:M\times N\longrightarrow T$ is called the Lie superpairing if the following relations are satisfied:
\begin{enumerate}
	\item $f([m , m'], n)= f(m, {}^ {m'}{n}) -(-1)^{|m||m'|}f(m', {}^{m}{n})$,
	\item  $f(m, [n,n'])=(-1)^{|n'|(|m|+|n|)}f({}^{n'}{m}, n)-(-1)^{|m||n|}f({}^{n}{m}, n')$,
	\item $f({}^{n}{m} , {}^{m'}{n'})=-(-1)^{|m||n|}[f(m, n),f(m', n')]),$  
\end{enumerate} 
for every $m, m' \in M_{\overline{0}}\cup  M_{\overline{1}}$ and $n , n' \in N_{\overline{0}}\cup  N_{\overline{1}}$.

Let $M$ and $N$ be two Lie superalgebras with actions on each other. Let $X_{M , N}$ be the $\mathbb{Z}_{2}$-graded set of all symbols $m\otimes n$, where $m \in M_{\overline{0}}\cup  M_{\overline{1}}$, $n \in N_{\overline{0}}\cup  N_{\overline{1}}$ and the $\mathbb{Z}_{2}$-gradation is given by $|m\otimes n|=|m|+|n|$. The non-abelian tensor product of $M$ and $N$, denoted by $M \otimes N$, as the Lie superalgebra generated by $X_{M , N}$  and subject to the relations:
\begin{enumerate} 
	\item $\lambda (m \otimes n)=\lambda m \otimes n= m \otimes \lambda n$,
	\item $(m + m')\otimes n= m\otimes n +m'\otimes n$,  where $m , m'$ have the same degree,\\
	$m \otimes (n + n')=m \otimes n + m \otimes n'$,  where $n , n'$ have the same degree,
	
	\item $[m , m']\otimes n= (m\otimes {}^ {m'}{n}) -(-1)^{|m||m'|}(m'\otimes {}^{m}{n})$,\\
	$m\otimes [n,n']=(-1)^{|n'|(|m|+|n|)}({}^{n'}{m}\otimes n)-(-1)^{|m||n|}({}^{n}{m}\otimes n')$,
	
	\item $[m\otimes n,m'\otimes n']=-(-1)^{|m||n|}({}^{n}{m} \otimes {}^{m'}{n'}),$      
\end{enumerate}
for every $\lambda \in \mathbb{K}, m, m' \in M_{\overline{0}}\cup  M_{\overline{1}}$ and $n , n' \in N_{\overline{0}}\cup  N_{\overline{1}}$. The tensor product $M \otimes N$ has $\mathbb{Z}_{2}$-grading given by $(M \otimes N)_{\alpha}=\oplus_{\beta+\gamma=\alpha}  (M_{\beta}+N_{\gamma})$ for $\alpha, \beta, \gamma \in \Z_2$. If $M=M_{\overline{0}}$ and $N=N_{\overline{0}}$ then $M \otimes N$ is the non-abelian tensor product of Lie algebras which was introduced and studied in \cite{Ellis1991}.
\smallskip

Actions of Lie superalgebras $M$ and $N$ on each other are said to be compatible if 
\begin{enumerate}
	\item ${}^{({}^{n}{m})}{n'}=-(-1)^{|m||n|}[{}^{m}{n},n']$,
	\item ${}^{({}^{m}{n})}{m'}=-(-1)^{|m||n|}[{}^{n}{m},m']$,
\end{enumerate}
for all $m, m' \in M_{\overline{0}}\cup  M_{\overline{1}}$ and $n, n' \in N_{\overline{0}}\cup  N_{\overline{1}}$. For instance if $M$, $N$ are two graded ideals of a Lie superalgebra then the actions induced by the bracket are compatible.

\begin{lemma}\label{prop3.1}\cite[Proposition 3.4(ii)]{GKL2015}
Let $M$ and $N$ be two Lie superalgebras with actions on each other. Then there are actions of both $L$ and $M$ on $L\otimes M$ given by
$${}^{l'} (l\otimes m) = [l',l]\otimes m +(-1)^{|l||l'|}~  l\otimes{}^{}({}^{l'}m),$$
$$ {}^{m'} (l\otimes m) = {}^{}({}^{m'}l)\otimes m +(-1)^{|m||m'|}~  l\otimes [m',m],$$ for all $l,~l' \in L_{0}\cup L_{1} $ and $m,~m' \in M_0 \cup M_1$.
\end{lemma}

From the above proposition, $L$ acts on $L \otimes L$. On the other hand, the tensor product $L \otimes L$ acts on $L$ by ${}^tl = {}^{\lambda(t)}l$ for all $t \in  L \otimes L$ and $l \in L$ such that $\lambda : L \otimes L \longrightarrow L$ is a homomorphism given by $a \otimes b \longmapsto [a, b]$. These actions are
compatible, and thus we can construct the triple tensor product $\otimes^3L = (L \otimes L) \otimes L$.

\begin{lemma}\label{p00}\cite[Proposition 3.5]{GKL2015}
If the Lie superalgebras $M$ and $N$ act trivially on each other, then $M\otimes N$ is an abelian Lie superalgebra and there is an isomorphism of supermodules 
  	\[M\otimes N \cong M^{ab}\otimes _{\mathbb{K}}N^{ab},\]
where $M^{ab}=M/[M, M]$ and $N^{ab}=N/[N, N]$. 
\end{lemma}

\begin{lemma}\label{prop4}\cite[Proposition 3.8]{GKL2015}
Given a short exact sequence of Lie superalgebras
 $$(0, 0) \longrightarrow (K,L) \overset{(i, j)} \longrightarrow (M, N) \overset{(\phi, \psi)} \longrightarrow (P, Q) \longrightarrow (0, 0)$$ there is an exact sequence of Lie superalgebras 
 $$(K \otimes N) \rtimes (M \otimes L) \overset{\alpha} \longrightarrow M \otimes N \overset{\phi \otimes \psi} \longrightarrow P \otimes Q \longrightarrow 0.$$
 \end{lemma}

Specifically given a Lie superalgebra $M$ and a graded ideal $K$ of $M$ there is an exact sequence 
 \begin{equation}\label{eq1}
 	(K \otimes M) \rtimes (M \otimes K) \longrightarrow M \otimes M \longrightarrow (M/K) \otimes (M/K) \longrightarrow 0.
 \end{equation}


\smallskip

Consider a Lie superalgebra $M = M_{\bar{0}}\oplus M_{\bar{1}}$ and let $M \square M$ be the submodule of $M \otimes M$ generated by elements
\begin{enumerate}
	\item $m \otimes m'+   (-1)^{|m||m'|}(m'\otimes m)$, 
	\item $m_{\overline{0}} \otimes m_{\overline{0}}$,
\end{enumerate} 
with $m,m' \in M_{\overline{0}} \cup M_{\overline{1}},~m_{0} \in M_{\overline{0}}$. Then the exterior product of $M$ and $M$ is denoted as $M \wedge M$ and is defined as the quotient Lie superalgebra \[M\wedge M= \frac{M\otimes M}{ M\square M}.\] 
For any $m \otimes m' \in M \otimes M$ we denote the coset $m \otimes m' +M \square M$ by $m \wedge m'$.\\
\smallskip

The exterior center of a Lie superalgebra $L$ is defined as follows:
\[Z^{\wedge}(L)=\{x \in L \mid x \wedge y =0, ~\forall~ y \in L \}.\]

\begin{lemma}\label{cor2}\cite[Corollary 5.10]{Padhandetec}
$ N\subseteq Z^{\wedge}(L)$ if and only if the natural map $L\wedge L \longrightarrow L/N\wedge L/N$ is a monomorphism.
\end{lemma}

\begin{lemma}\label{lemma3}\cite[Lemma 5.9, Proposition 5.7]{Padhandetec}
	Let $N$ be a central graded ideal of Lie superalgebra $L$ then,
\begin{enumerate}
	\item 	$L\wedge N\longrightarrow L\wedge L \longrightarrow L/N\wedge L/N \longrightarrow 0$,
   \item  $ \mathcal{M}(L)  \overset{\sigma} \longrightarrow \mathcal{M}(L/N)\longrightarrow N\cap L^{2}\longrightarrow 0$,
\end{enumerate}
are exact sequences.

\end{lemma}
\begin{lemma}\label{lemma29}
Let $L$ be a Lie superalgebra. Then 
	\[0\longrightarrow  ~\mathcal{M}(L) \longrightarrow L\wedge L \longrightarrow L^{\prime}\longrightarrow 0 \]
is a central extension.
\end{lemma}

\begin{lemma}\label{corr3.5}
Let $L$ be a nilpotent Lie superalgebra of class two. Then $L\wedge L$ is an abelian. Moreover $L\wedge L\cong \mathcal{M}(L)\oplus L^2$.
\end{lemma}
\begin{proof}
Since $[l_1 \wedge l_2, l_3 \wedge l_4] =-(-1)^{|l_1||l_2|} [l_2, l_1] \wedge [l_3, l_4] = -(-1)^{|l_1||l_2|} ((l_2 \wedge [l_1, [l_3, l_4]]) -(-1)^{|l_1||l_2|} (l_1 \wedge [l_2, [l_3, l_4]])) =(l_1 \wedge [l_2, [l_3, l_4]])-(-1)^{|l_1||l_2|} (l_2 \wedge [l_1, [l_3, l_4]])= 0,$ for all $l_1,~ l_2,~ l_3,~ l_4 \in L$. Thus $(L\wedge L)^2 = 0$, and $L\wedge L$ is abelian and by invoking Lemma \ref{lemma29}, we have $L\wedge L\cong \mathcal{M}(L)\oplus L^2$.
\end{proof}

\begin{lemma}\label{th4.4}\cite[Proposition 3.4]{SN2018b}\cite[Theorem 6.9]{Padhandetec}
Let $L$ be a nilpotent Lie superalgebra of dimension $(k \mid l)$ with $\dim L'=(r \mid s)$, where $r+s=1$. If $r=1, s=0$, then $L \cong H(m,n)\oplus A(k-2m-1 \mid l-n)$ for $m+n\geq 1$. If $r=0, s=1$, then $L \cong H_{m} \oplus A(k-m \mid l-m-1)$. Moreover, $L$ is capable if and only if either $L \cong H(1 , 0)\oplus A(k-3 \mid l)$ or $L \cong H_{1}\oplus A(k-1 \mid l-2)$.
\end{lemma}

\begin{lemma}\label{cor3}\cite[Corollary 5.12, Lemma 6.2]{Padhandetec}
	\begin{enumerate}
    	\item 	$ \mathcal{M}(A(m \mid n)) \cong A(m \mid n)\wedge A(m \mid n)$.
		\item $H(1 , 0)\wedge H(1 , 0) \cong A(3 \mid 0)$
		\item $H(0 , 1)\wedge H(0 , 1)\cong A(1 \mid 0)$
		\item $H(m , n)\wedge H(m , n)\cong A(r \mid s)$, where $r=(2m^{2}-m)+\dfrac{n(n+1)}{2}, s= 2mn$ and $m+n\geq 2$.
		\item  $H_{1}\wedge H_{1} \cong A(1 \mid 2)$
		\item  $H_{m}\wedge H_{m}\cong A(m^{2} \mid m^{2})$ for $m \geq 2$.
	\end{enumerate}
\end{lemma}

\begin{lemma}\label{prop444}\cite[Proposition 4.4]{hasanpadprad}
Let $H(m,n)$ be a Heisenberg Lie superalgebra with an even center, then
	\[H(m,n)\otimes H(m,n)\cong H(m,n)/H^2(m,n) \otimes H(m,n)/H^2(m,n),\]
when $m+n\geq 2.$ Moreover, $H(1, 0) \otimes H(1, 0) \cong A (6 \mid 0)$ and $H(0, 1) \otimes H(0, 1) \cong A(1 \mid 0)$.
\end{lemma}

\begin{lemma}\label{prop45}\cite[Proposition 4.5]{hasanpadprad}
Let $H_m$ be a Heisenberg Lie superalgebra with odd center, then \[H_m\otimes H_m\cong H_m/H^2_m \otimes H_m/H^2_m,\] when $m\geq 2.$ Moreover, $H_1 \otimes H_1 \cong A (2 \mid 3)$.
\end{lemma}

\begin{lemma}\label{cor4.3}
Let $L$ be a finite dimension Lie superalgebra, then
	\[L \square L\cong L/L^2\square L/L^2\cong \Gamma(L/L^2).\]	
\end{lemma}

\begin{lemma}\label{lem3.9}
Let $L$ be a Lie superalgebra such that $L/L^2$ is of finite dimension and $\pi:L\otimes L\longrightarrow L/L^2 \otimes L/L^2$ be the natural epimorphism. Then the restriction $\pi :L\square  L\longrightarrow L/L^2 \square  L/L^2$is an isomorphism.
\end{lemma}

\begin{lemma}\label{lem3.10}
Let $L$ be an abelian Lie superalgebra with a basis $\{x_1, x_2, \ldots , x_{m+n}\}$, where $|x_i|=0, 1\leq i\leq m$ and $|x_{m+j}|=1, 1< j\leq n$. Then
	$$L \otimes L\cong L\square L \oplus   \left\langle x_i \otimes x_j ,x_l \otimes x_l ~ \lvert ~ 1 \leq i < j \leq m+n, \, \, m+1 \leq l \leq m+n  \right\rangle .$$
\end{lemma}

\begin{proof}
Since $L\cong\bigoplus^{m+n}_{i=1}\left\langle x_i\right\rangle $, we have
$L \otimes L \cong <  x_i \otimes x_j +(-1)^{|x_i||x_j|} (x_j \otimes x_i) ,~ x_r \otimes x_r ~|~ 1 \leq i < j \leq m+n, \, \, 1\leq r \leq m > \oplus < x_i \otimes x_j , ~ x_l \otimes x_l  ~| ~1 \leq i < j \leq m+n, \, \, m+1 \leq l \leq m+n   >.$ Since  $L\square L \cong \left\langle  x_i \otimes x_j +(-1)^{|x_i||x_j|} (x_j \otimes x_i) ,~ x_r \otimes x_r ~| ~1 \leq i < j \leq m+n, \, \, 1\leq r \leq m \right\rangle$, thus the result follows.
\end{proof}

\begin{lemma}\label{lem3.12}
Let	$L/L^2$ be a finite dimensional Lie superalgebra. Then
	$$L\otimes L  \cong L\wedge L \oplus L \square L.$$	
\end{lemma}

\begin{proof}
From Lemma \ref{lem3.10}, we have 
$$L/L^2 \otimes L/L^2 \cong L/L^2\square L/L^2 \oplus \left\langle \overline{x}_i\otimes \overline{x}_j , ~ \overline{x}_l \otimes \overline{x}_l ~ |~  1 \leq i < j \leq r+s, \, \, r+1 \leq l \leq r+s   \right\rangle, $$
where $\{\overline{x}_1, \ldots , ¯\overline{x}_{r+s}\}$ is a basis for $L/L^2$. Now consider the map $\pi :L\otimes L\longrightarrow L/L^2 \otimes L/L^2$, then 
$$\pi(L\square L + \left\langle x_i \otimes x_j ,~x_l \otimes x_l ~\lvert ~1 \leq i < j \leq m+n, \, \, m+1 \leq l \leq m+n  \right\rangle )$$  $$= L/L^2\square L/L^2 \oplus \\ \left\langle \overline{x}_i\otimes \overline{x}_j,  \overline{x}_l \otimes \overline{x}_l  | \,\, 1 \leq i < j \leq r+s, \, \, r < l \leq r+s   \right\rangle. $$
Hence $L\otimes L= L\square L + (\left\langle x_i \otimes x_j ,x_l \otimes x_l \lvert 1 \leq i < j \leq m+n, \, \, m < l \leq m+n  \right\rangle + \ker \pi).$ Now, from Lemma \ref{lem3.9}, $\pi$ is an isomorphism, and it
maps $x$ to zero, where $$x \in L\square L ~\cap~ (\left\langle x_i \otimes x_j ,~x_l \otimes x_l ~\lvert ~1 \leq i < j \leq m+n,~ m+1 \leq l \leq m+n  \right\rangle + \ker \pi)=L\square L \cap ker\pi, $$ which implies that $L\otimes L= L\square L \oplus (\left\langle x_i \otimes x_j ,~x_l \otimes x_l ~\lvert~ 1 \leq i < j \leq m+n, \, \, m+1 \leq l \leq m+n  \right\rangle + \ker \pi).$ Now the proof follows, as $ \left\langle x_i \otimes x_j ,~x_l \otimes x_l ~\lvert~ 1 \leq i < j \leq m+n, \, \, m+1 \leq l \leq m+n  \right\rangle + \ker \pi \cong \frac{L\otimes L}{L \square L}= L\wedge L$. 

\end{proof}

The universal quadratic functor of Lie superalgebra  was  introduced in \cite{Pilar}. The quadratic functor  helps to establish the relations between the Lie exterior product and the Lie tensor product of the Lie superalgebras. Here we recall the definition of the universal quadratic functor of Lie superalgebra and some results from \cite{Pilar} which will be useful in the next section.

\begin{definition}
Let $L$ be a supermodule. We define the supermodule $\Gamma(M)$ as the direct sum
	\[\Gamma(M)=R^{M_{\bar{0}}} \oplus (M\otimes M) 
	=\{\gamma(m_{\bar{0})}+ m'\otimes m'' \lvert m_{\bar{0}} \in M_{\bar{0}}, \, m',m'' \in M\},\]
and subject to the homogeneous relations
	\begin{equation}
		\gamma(\lambda m_{\bar{0}})=\lambda^2\gamma(m_{\bar{0}})
	\end{equation}
	
	\begin{equation}
		\gamma(m_{\bar{0}}+ m'_{\bar{0}})-\gamma(m_{\bar{0}})-\gamma(m'_{\bar{0}})= m_{\bar{0}}\otimes m'_{\bar{0}}
	\end{equation}
	\begin{equation}
		m\otimes m'= (-1)^{|m||m'|} m'\otimes m
	\end{equation}

	\begin{equation}
		m_{\bar{1}}\otimes m_{\bar{1}}=0
	\end{equation}
where $\lambda \in \mathbb{K}$, $m_{\bar{0}}, m'_{\bar{0}} \in M_{\bar{0}}$, $	m_{\bar{1}} \in 	M_{\bar{1}}$ and $m,m' \in M$ with the induced grading.
\end{definition}

\begin{lemma}\label{p111}
For any Lie superalgebra $M$, there exists an exact sequence
	\begin{equation}\label{eq31}
		\Gamma(M^{ab}) \overset{\psi}\longrightarrow M\otimes M \overset{\pi}\longrightarrow M \wedge M \longrightarrow 0.
	\end{equation}
Also, given two graded ideals $I$ and $J$ of $M$, the following sequence is exact:
\begin{equation}
	\Gamma(\frac{I\cap J}{[I,J]}) \overset{\psi}\longrightarrow I\otimes J \overset{\pi}\longrightarrow I \wedge J \longrightarrow 0.
\end{equation}
\end{lemma}

\section{Main results}
A Lie superalgebras is said to be a generalized Heisenberg Lie superalgebras of rank $(m\mid n)$ if $L'=Z(L)$ and $\dim Z(L)=(m \mid n)$. In this section, we intend to give the explicit structure of $L\otimes L = \otimes^2L$ and $L\wedge L  = \wedge^2L$ when $L$ is a generalized Heisenberg of rank 2. Moreover, we obtain $\otimes^3L$ and $\wedge^3L$ when  $L$ is a generalized Heisenberg Lie superalgebra of rank either $0,~1,$ and $2$. Also, for a non-abelian nilpotent Lie superalgebra $L$, we give an upper bound for the triple tensor product of $L$. Now first we calculate the Schur multiplier of non-capable generalized Heisenberg Lie superalgebra of rank 2.

\begin{theorem}\label{lemm4.1}
Let $H$ be a non-capable generalized Heisenberg Lie superalgebra of rank 2 such that $\dim H = (m\mid n)$. Then we have the following cases: \\
\begin{enumerate}
\item  If $Z^{\wedge}(H)=H^2$, then
$$ \dim\mathcal{M}(H)=
\begin{cases}
(\frac{1}{2}(m(m-1)+(n-2)(n-1))\mid m(n-2)-2)\quad \mbox{if}\;\dim H^2=(0\mid 2)\\
	 		(\frac{1}{2}((m-3)(m-2)+n(n+1)-4) \mid (m-2)n) \quad \mbox{if}\;\dim H^2=(2\mid 0)\\
(\frac{1}{2}((m-1)(m-2)+n(n-1)-2) \mid (m-1)(n-1)-1) \quad \mbox{if}\;\dim H^2=(1\mid 1).\\  
\end{cases}$$

\item  If $Z^{\wedge}(H)=K\subseteq H^2$ and $\dim K=(1 \mid 0)$, then
$$\dim\mathcal{M}(H)=
\begin{cases}
(\frac{1}{2}((m-2)(m+1)+(n-2)(n-1)) \mid m(n-2)+1)\quad \mbox{if}\;\dim H^2=(1\mid 1)\\
(\frac{1}{2}((m-4)(m-1)+n(n+1)+2) \mid (m-2)n) \quad \mbox{if}\;\dim H^2=(2\mid 0).\\  
\end{cases}$$
	
\item  If $Z^{\wedge}(H)=K\subseteq H^2$ and $\dim K=(0 \mid 1)$, then 
$$ \dim\mathcal{M}(H)=
 \begin{cases}
(\frac{1}{2}((m-4)(m-1)+n(n+1)+4) \mid (m-2)n-1)\quad \mbox{if}\;\dim H^2=(1\mid 1)\\
(\frac{1}{2}((m+2)(m-1)+(n-3)(n-2)+2) \mid (m+1)(n-3)) \quad \mbox{if}\;\dim H^2=(0\mid 2).\\  
\end{cases}$$
\end{enumerate}
\end{theorem}

\begin{proof}
\textbf{Case 1.} Since $H$ is non-capable, by Lemma \ref{cor2} and Lemma \ref{lemma3}(ii), we have $\dim\mathcal{M}(H) = \dim\mathcal{M}(H/H^2) -(0\mid 2)$,
and so by Lemma \ref{th3.3},
\begin{align*}
		\dim\mathcal{M}(H)
		& =\dim\mathcal{M}(A(m \mid n-2)) -(0\mid 2)\\
		& =(\frac{1}{2}((m(m-1)+(n-2)(n-1))\mid m(n-2)-2).\\
\end{align*}
Similarly, it is easy to conclude when $\dim H=(2\mid 0)$ and $\dim H=(1\mid 1)$.\\
	
\textbf{Case 2.} Let $Z^{\wedge}(H)=K\subseteq H^2$, $\dim H^2=(2\mid 0)$, and $\dim K=(1 \mid 0)$. Since $H$ is non-capable, by invoking Lemma \ref{cor2} and Lemma \ref{lemma3}(ii) we have $\dim \mathcal{M}(H)=\dim \mathcal{M}(H/Z^{\wedge}(H))-(1\mid 0).$ As $H/Z^{\wedge}(H)$ is capable and $\dim(H/Z^{\wedge}(H))^2=(1\mid 0)$, Lemma \ref{th4.4} implies that $H/Z^{\wedge}(H)\cong H(1\mid0)\oplus A(m-4\mid n)$. Now, by Lemma \ref{th3.7}, we get
\begin{align*}
		\dim \mathcal{M}(H/Z^{\wedge}(H))
		&= \dim \mathcal{M}(H(1,\mid 0)+ \dim \mathcal{M}(A(m-4\mid n))+\dim(H(1\mid 0)/H^2(1\mid0)\otimes A(m-4\mid n))\\
		&=	(\frac{1}{2}((m-4)(m-1)+n(n+1)+4) \mid (m-2)n).\\
\end{align*}
Hence $\dim \mathcal{M}(H)=	(\frac{1}{2}((m-4)(m-1)+n(n+1)+2) \mid (m-2)n).$ Similarly, we can prove when $\dim H=(1\mid 1)$.\\
	
\textbf{Case 3.} The proof is similar to the proof of Case-2.
\end{proof}

Let  $\gamma_k(L)$ be the $k$-th term of the descending central series of $L$ and the tensor product $L \otimes L$ acts on $L$ by ${}^tl = {}^{\lambda(t)}l$ where $\lambda : L \otimes L \longrightarrow L$ is a homomorphism defined by $a \otimes b \longmapsto [a, b]$. Then we have the following results on $(L \otimes L) \otimes L$ and $\otimes^3L$.

\begin{lemma}\label{lem4.22}
Let $L$ be a Lie superalgebra of nilpotency class two. Then
\begin{enumerate}
		\item $L \otimes L$ acts trivially on $L$.
		\item $(L \otimes L) \otimes L$ is an abelian Lie superalgebra.
\end{enumerate}
\end{lemma}
\begin{proof}
Since nilpotency class of $L$ is two, so $\gamma_3(L) = 0$. Thus  $L\otimes L$ acts trivially on $L$, as ${}^{\lambda(a\otimes b)}c = [[a, b], c] = 0$ for all $a\otimes b \in L\otimes L$ and $c\in  L$. Second part follows as

\begin{align*}
	[(a \otimes b) \otimes c, (a' \otimes b')\otimes c']
	& =-(-1)^{ |a \otimes b| |c|} \,\, {}^{c}{(a\otimes b)}\otimes {}^{(a'\otimes b')}{c'}\\
	&=-(-1)^{( |a|+ |b|) |c|} \,\, {}^{c}{(a\otimes b)}\otimes 0\\
	&=0.
\end{align*}

\end{proof}

Let $L$ be a nilpotent Lie superalgebra of class $k$. Let $i_L : L \longrightarrow L$ be the identity homomorphism and $\varphi : \gamma_k(L) \longrightarrow L$
be a natural embedding. Define homomorphisms $\overline{\varphi} = (\varphi \otimes i_L) \otimes i_L : (\gamma_k(L) \otimes L) \otimes L\longrightarrow \otimes^3L$ and $\gamma : (L \otimes L) \otimes \gamma_k(L)\longrightarrow \otimes^3L$ given by $(a \otimes b) \otimes c\longmapsto (a \otimes b)\otimes c$.
Then we have the following results.

\begin{lemma}\label{lemm4.2}
If $L$ is a non-abelian nilpotent Lie superalgebra of class $k$, then $Im\gamma \subseteq Im\overline{\varphi}$.
\end{lemma}
\begin{proof}
As $k \otimes [l, l'] =-(-1)^{( |k|+ |l|) |l'|} ([l',k]\otimes l)-(-1)^{|k||l|} ([l,k]\otimes l')$ on $L \otimes L$, therefore $(a \otimes b) \otimes [x_1, \ldots , x_{k-1}, x_k] \in (\gamma_k(L) \otimes L) \otimes L $ for all $a, b, x_1, \ldots , x_{k-1}, x_k \in L$. Thus the result follows.
\end{proof}

\begin{lemma}\label{prop4.44}
If $L$ is a nilpotent Lie superalgebra of class $k$, then
	$$(\gamma_k(L) \otimes L) \otimes L
\xrightarrow{(\varphi\otimes i_L)\otimes i_L} \otimes^3L \longrightarrow \otimes^3L/\gamma_k(L) \longrightarrow 0,$$
is exact.
\end{lemma}

\begin{proof}
Since $\gamma_k(L)\longrightarrow L\longrightarrow L/\gamma_k(L)\longrightarrow 0$
is exact, thus by Lemma \ref{prop4} we get the following exact sequence 
\[ \gamma_k(L)\otimes L\longrightarrow L\otimes L\longrightarrow L/\gamma_k(L)\otimes L/\gamma_k(L) \longrightarrow 0.\]
Now the result follows by again using Lemma \ref{prop4} and Lemma \ref{lemm4.2}. 
\end{proof}

In the next two theorems, we determine the triple tensor product and triple exterior product of Heisenberg Lie superalgebra with even as well as odd center.

\begin{theorem}\label{the4.5}
Let $L$ be a Heisenberg Lie superalgebra with even center. Then
$$
	\otimes^3H(m,n)\cong
	\begin{cases}
		A(12 \mid 0)\quad \mbox{if}\;m=1, n=0\\
		A(0 \mid 1) \quad \mbox{if}\;m=0, n=1\\  
		A(8m^3+6mn^2 \mid 12m^2n+n^3)\quad \mbox{if}\;m+n\geq 2.
	\end{cases}
	$$
	and

	$$
	\wedge^3H(m,n)\cong
	\begin{cases}
		A(2 \mid 0))\quad \mbox{if}\;m=1,n=0\\
		A(0 \mid 1) \quad \mbox{if}\;m=0, n=1\\  
		A(4m^3-4m^2+mn(n+1)+2mn^2-m-\frac{n(n-1)}{2} \mid 6m^2n-3mn +\frac{n^2(n+1)}{2})\quad \mbox{if}\;m+n\geq 2.
	\end{cases}
	$$

\end{theorem}

\begin{proof}
First we show that $(\varphi(L^2)\otimes L)\otimes L \cong (\varphi(L^2)\otimes L^{ab})\otimes L^{ab}$. As $\varphi(L^2)$ and $L$ act trivially on each other, thus by Lemma \ref{p00}, we have $\varphi(L^2) \otimes L \cong \varphi(L^2) \otimes L^{ab}$. Again $\varphi(L^2) \otimes L^{ab}$ and $L$ act trivially on each other. Thus $(\varphi(L^2) \otimes L) \otimes L \cong (\varphi(L^2)\otimes L^{ab}) \otimes L^{ab}$. From the following exact sequence
	\[ L^2\otimes L\overset{\varphi\otimes i_L}\longrightarrow L\otimes L\longrightarrow L/L^2\otimes L/L^2 \longrightarrow 0,\]
we get 
\begin{eqnarray}\label{eq4.1}
\dim L\otimes L = \dim L^{ab} \otimes L^{ab}+\dim Im\varphi\otimes i_L = \dim L^{ab}\otimes L^{ab}+\dim \varphi(L^2)\otimes L.
\end{eqnarray}

\textbf{Case 1.} Let $L \cong H(1\mid 0)$. Since $H(1\mid 0) \otimes H(1\mid 0) \cong A(6\mid 0)$ and $H(1\mid 0)^{ab} \otimes H(1\mid 0)^{ab}\cong A(4\mid 0)$. Now from  \ref{eq4.1}, $\dim\varphi(L^2) \otimes L = (2\mid 0)$. As $(\varphi(L^2) \otimes L) \otimes L \cong (\varphi(L^2) \otimes L^{ab}) \otimes L^{ab}$, so
	\begin{eqnarray}
		\dim(\varphi(L^2) \otimes L) \otimes L = \dim (\varphi(L^2) \otimes L^{ab}) \dim L^{ab}=\dim (\varphi(L^2) \otimes L) \dim L^{ab}.
	\end{eqnarray}
Now using Lemma \ref{prop4.44}
\begin{align*}
	\dim \otimes^3L &=\dim \otimes^3L^{ab} + \dim (\varphi(L^2) \otimes L) \dim L^{ab}\\
	& =(8\mid 0)+ (2\mid 0)\times(2\mid 0)\\
	&=(12\mid 0).
\end{align*}
Also, if we use Lemma \ref{lem4.22}, then  $\otimes^3L\cong A(12\mid 0).$ \\

\textbf{Case 2.} Let $L=H(m\mid n)$, where $m+n\geq 2$. Similarly we can find $	\dim(\varphi(L^2) \otimes L)=0$ and thus $$	\dim \otimes^3L =\dim \otimes^3L^{ab}=(8m^3+6mn^2 \mid 12m^2n+n^3).$$ Then by Lemma \ref{lem4.22}, $\otimes^3L\cong A(8m^3+6mn^2 \mid 12m^2n+n^3)$.

\textbf{Case 3.} Let $L=H(0\mid 1)$. Then similar to Case-2, we can find that $	\dim \otimes^3L =A(0\mid 1).$\\

\textbf{Case 4.} Let $L \cong H(1\mid 0)$. From the sequence
\[  L^2 \wedge L\overset{\varphi\wedge i_L}\longrightarrow L\wedge L\longrightarrow L/L^2 \wedge L/L^2 \longrightarrow 0  \]
and Lemma \ref{cor3}, we get $\dim \varphi(L^2)\wedge L = \dim \wedge^2L - \dim \wedge^2L^{ab} =(3\mid 0)- (1\mid 0)= (2\mid 0)$. Observe that $(\varphi(L^2)\wedge L)\wedge L \cong (\varphi(L^2)\wedge L^{ab})\wedge L^{ab}$. Since $\varphi(L^2)\wedge L \cong A(2\mid 0)$, we have $(\varphi(L^2)\wedge L)\wedge L \cong A(2\mid 0)\wedge A(2\mid 0)$. Now from Lemma \ref{cor3} $(\varphi(L^2)\wedge L)\wedge L \cong \mathcal{M}(A(2\mid 0))$. By invoking Lemma \ref{th3.3} we have  $(\varphi(L^2)\wedge L)\wedge L\cong A(1\mid 0).$
\smallskip

On the other hand, $\wedge^3L/L^2 \cong (A(2\mid 0) \wedge A(2\mid 0)) \wedge A(2\mid 0)$, hence if we use Lemma \ref{cor3}, Lemma \ref{c001} and Lemma \ref{c002}, then $\wedge^3L/L^2 \cong A(1\mid 0)$. 
 Consider the following
exact sequence
	$$(L^2 \wedge L) \wedge L
\xrightarrow{(\varphi\wedge i_L)\wedge i_L} \wedge^3L \longrightarrow \wedge^3L/L^2 \longrightarrow 0,$$

Then 
\begin{align*}
	\dim \wedge^3L &= \dim \wedge^3 L/L^2 + \dim Im(\varphi\wedge i_L)\wedge i_L\\
	& = \dim\wedge^3L/L^2 + \dim(\varphi(L^2) \wedge L) \wedge L\\
	&=(1\mid 0)+ (1\mid 0)=(2\mid 0).	
\end{align*}
Lemma \ref{lem4.22} (ii) tells us that $\otimes^3L$ is an abelian, thus $\wedge^3L \cong A(2\mid 0)$.\\

\textbf{Case 5.} Let $L \cong H(0\mid 1)$, then similarly, one can obtain that $ \varphi(L^2)\wedge L =0$ and if we take $M=A(1|1),~I=A(0|1),~J=A(1|0)$ in Lemma \ref{p111}, then we have $ \wedge^3L=A(0|1).$\\

\textbf{Case 6.} Let $L \cong H(m\mid n)$, where $m+n\geq 2$. Then similar to Case-4, one can determine $\varphi(L^2)\wedge L =0$ and $\wedge^3H(m,n)=A(r \mid s)$, where $r=4m^3-4m^2+mn(n+1)+2mn^2-m-\frac{n(n-1)}{2}$ and  $s=6m^2n-3mn +\frac{n^2(n+1)}{2}$.
\end{proof}

Similar to the above theorem we state for the odd center without proof, as the proof runs quite similar to the previous theorem.

\begin{theorem}
Let $L$ be a Heisenberg Lie superalgebra with odd center. Then
$$
\otimes^3H_m\cong
\begin{cases}
	A(5 \mid 5))\quad \mbox{if}\;m=1\\  
	A(4m^2 \mid 4m^2)\quad \mbox{if}\;m\geq 2.
\end{cases}
$$
and 
$$
\wedge^3H_m\cong
\begin{cases}
	A(2 \mid 2))\quad \mbox{if}\;m=1\\
	A(2m^3-m^2 \mid m^4-(m^2-m)^2)\quad \mbox{if}\;m\geq 2.
\end{cases}
$$

\end{theorem}

In the following theorem we calculate the dimension and the structure of tensor product and exterior product of non-capable generalized Heisenberg Lie superalgebra of rank 2.

\begin{theorem}
Let $H$ be an $(m\mid n)$-dimensional non-capable generalized Heisenberg Lie superalgebra of rank $(r \mid s)$ where $r+s=2$. Then we have the following cases:
	 \begin{enumerate}
			\item  If $Z^{\wedge}(H)=H^2$, then
		
		$$
		\wedge^2H=
		\begin{cases}
			A(\frac{1}{2}(m(m-1)+(n-2)(n-1))\mid m(n-2))\quad \mbox{if}\;\dim H^2=(0\mid 2)\\
			A(\frac{1}{2}((m-3)(m-2)+n(n+1)) \mid (m-2)n) \quad \mbox{if}\;\dim H^2=(2\mid 0)\\
		A(\frac{1}{2}((m-1)(m-2)+n(n-1)) \mid (m-1)(n-1)) \quad \mbox{if}\;\dim H^2=(1\mid 1).\\  
			
		\end{cases}$$
			$$
		\otimes^2H=
		\begin{cases}
			A(m^2+(n-2)^2\mid 2m(n-2))\quad \mbox{if}\;\dim H^2=(0\mid 2)\\
			A((m-2)^2+n^2 \mid 2(m-2)n) \quad \mbox{if}\;\dim H^2=(2\mid 0)\\
			A((m-1)^2+(n-1)^2 \mid 2(m-1)(n-1)) \quad \mbox{if}\;\dim H^2=(1\mid 1).\\  
			
		\end{cases}$$
			\smallskip
		\item  If $Z^{\wedge}(H)=K\subseteq H^2$ and $\dim K=(1 \mid 0)$, then
		$$
		\wedge^2H=
		\begin{cases}
			A(\frac{1}{2}((m-2)(m+1)+(n-2)(n-1)+2) \mid m(n-2)+2)\quad \mbox{if}\;\dim H^2=(1\mid 1)\\
			A(\frac{1}{2}((m-4)(m-1)+n(n+1)+6) \mid (m-2)n) \quad \mbox{if}\;\dim H^2=(2\mid 0).\\  
			
		\end{cases}$$
		
		$$
		\otimes^2H=
		\begin{cases}
			A(l+\frac{1}{2}m(m-1) \mid k+(m-1)(n-1))\quad \mbox{if}\;\dim H^2=(1\mid 1)\\
			A((m-3)(m-1)+n^2+3 \mid 2(m-2)n) \quad \mbox{if}\;\dim H^2=(2\mid 0),\\  
			
		\end{cases}$$
	where $l=\frac{1}{2}((m-2)(m+1)+2(n-2)(n-1)+2), $ and $k=m(n-2)+2.$
	\smallskip
		\item  If $Z^{\wedge}(H)=K\subseteq H^2$ and $\dim K=(0 \mid 1)$, then 
		$$
		\wedge^2H=
		\begin{cases}
			A(\frac{1}{2}((m-4)(m-1)+n(n+1)+6) \mid (m-2)n)\quad \mbox{if}\;\dim H^2=(1\mid 1)\\
			A(\frac{1}{2}((m+2)(m-1)+(n-3)(n-2)+2) \mid (m+1)(n-3)+2) \quad \mbox{if}\;\dim H^2=(0\mid 2).\\  
			
		\end{cases}$$
	
		$$
	\otimes^2H=
	\begin{cases}
		A((m-2)(m-1)+n(n-1)+4 \mid (m-1)(n-1)+(m-2)n)\quad \mbox{if}\;\dim H^2=(1\mid 1)\\
		A(r+\frac{1}{2}m(m+1) \mid s+m(n-2)) \quad \mbox{if}\;\dim H^2=(0\mid 2).\\  
		
	\end{cases}$$
where $r=\frac{1}{2}((m+2)(m-1)+2(n-3)(n-2)+2)$ and $s=(m+1)(n-3)+2.$

	\end{enumerate}
\end{theorem}
\begin{proof}
From Lemma \ref{corr3.5}, one can see that $\wedge^2H$ is abelian and $\wedge^2H\cong\mathcal{M}(H)\oplus H^2$. Hence the results for exterior product of $H$ follows from Lemma \ref{lemm4.1}. Now by using Lemma \ref{cor4.3},  we have $(H/H^2)\square(H/H^2)\cong H\square H\cong \Gamma(H/H^2)$, then 

	$$
H\square H=
\begin{cases}
	A(\frac{1}{2}m(m+1)+\frac{(n-2)(n-3)}{2}\mid m(n-2))\quad \mbox{if}\;\dim H^2=(0\mid 2)\\
	A(\frac{1}{2}(m-1)(m-2)+\frac{n(n-1)}{2} \mid (m-2)n) \quad \mbox{if}\;\dim H^2=(2\mid 0)\\
	A(\frac{1}{2}m(m-1) +\frac{(n-1)(n-2)}{2}\mid (m-1)(n-1)) \quad \mbox{if}\;\dim H^2=(1\mid 1).\\  
	
\end{cases}$$
As $\otimes^2H\cong H\wedge H\oplus H\square H$, thus the result follows.
\end{proof}

Now we determine the dimension and the structure of triple tensor product and exterior product of non-capable generalized Heisenberg Lie superalgebra of rank 2.

\begin{theorem}
	Let $H$ be an $(m\mid n)$-dimensional non-capable generalized Heisenberg Lie superalgebra
	of rank $(r \mid s)$, where $r+s=2$. Then we have the following cases:
	\begin{enumerate}
		\item  If $Z^{\wedge}(H)=H^2$, then
		
		$$
		\wedge^3H=
		\begin{cases}
			A(l\mid k)\quad \mbox{if}\;\dim H^2=(0\mid 2)\\
			A(s \mid d ) \quad \mbox{if}\;\dim H^2=(2\mid 0)\\
			A(x \mid y) \quad \mbox{if}\;\dim H^2=(1\mid 1).\\  
			
		\end{cases}$$

		$$
		\otimes^3H=
		\begin{cases}
			A(m+(m^2+3(n-2)^2)\mid (n-2)(3m^2+(n-2)^2))\quad \mbox{if}\;\dim H^2=(0\mid 2)\\
			A((m-2)((m-2)^2+3n^2) \mid 3n(m-2)^2+n^3) \quad \mbox{if}\;\dim H^2=(2\mid 0)\\
			A((m-1)((m-1)^2+3(n-1)^2) \mid (n-1)((n-1)^2+3(m-1)^2)) \quad \mbox{if}\;\dim H^2=(1\mid 1).\\  
			
		\end{cases}$$
		\smallskip
		\item If $Z^{\wedge}(H)=K\subseteq H^2$ and $\dim K=(1 \mid 0)$, then
		$$
		\wedge^3H=
		\begin{cases}
			A(x \mid y)\quad \mbox{if}\;\dim H^2=(1\mid 1)\\
			A(s+2m-7 \mid d+2n) \quad \mbox{if}\;\dim H^2=(2\mid 0).\\  
			
		\end{cases}$$
		
		$$
		\otimes^3H=
		\begin{cases}
			A((m-1)((m-1)^2+3(n-1)^2+1) \mid (n-1)((n-1)^2+3(m-1)^2+1))\quad \mbox{if}\;\dim H^2=(1\mid 1)\\
			A((m-2)((m-2)^2+3n^2+2) \mid n(n^2+3(m-2)^2+2)) \quad \mbox{if}\;\dim H^2=(2\mid 0),\\  
			
		\end{cases}$$
	
		\smallskip
		\item  If $Z^{\wedge}(H)=K\subseteq H^2$ and $\dim K=(0 \mid 1)$, then 
		$$
		\wedge^3H=
		\begin{cases}
			A(x+3(m-3) \mid y+3(n-1))\quad \mbox{if}\;\dim H^2=(1\mid 1)\\
			A(l+2m+n-5 \mid k+m+2n-6) \quad \mbox{if}\;\dim H^2=(0\mid 2).\\  
			
		\end{cases}$$
		
		$$
		\otimes^3H=
		\begin{cases}
			A((m-1)((m-1)^2+3(n-1)^2+3)\mid (n-1)((n-1)^2+(m-1)^2+3))\quad \mbox{if}\;\dim H^2=(1\mid 1)\\
			A(m(m^2+3(n-2)^2+1) \mid (n-2)(3m^2+(n-2)^2+1)) \quad \mbox{if}\;\dim H^2=(0\mid 2).\\  
			
		\end{cases}$$
			where $l=\frac{1}{2}(m(m^2-2m+9)+mn(3n-11)-(n-3)(n-2))$,\\ $k=\frac{1}{2}(3m(m-1)(n-2)+n^2(n-5)-4(n+1))$,\\ $s=\frac{1}{2}(m^2(m-8)+mn(3n+1)+19m-14-n(7n+1))$,
		$d=mn(m-12)+9n+\frac{1}{2}n^2(n+1)$,\\
		$x=\frac{1}{2}(m^2(m-5)-4n(n-2)+nm(3n-5)+8m-6)$, $y=\frac{1}{2}((n-1)(5m^2-15m+10-m^3+3m^2-2m+n^2-n)$.
	\end{enumerate}
\end{theorem}
\begin{proof}
The proof is similar to the proof of Theorem \ref{the4.5}.
\end{proof}

\begin{theorem}
Let $L$ be an $(m\mid n)$-dimensional non-abelian nilpotent Lie superalgebra with
derived subalgebra of dimension $(r\mid s)$. Then
	$$\dim\otimes^3L \leq (m+n)(m+n - (r+s))^2.$$
In particular, for $r=1,s=0$ the equality holds if and only if $L \cong H(1\mid 0)$.
\end{theorem}

\begin{proof}
Suppose $L$ is an $(m\mid n)$-dimesional non-abelian nilpotent Lie superalgebra of nilpotency class $k$. We prove by induction on $m+n$. Since $L$ is non-abelian, $m+n \geq 2$, from \cite{Hegazi} there is only one non-abelian nilpotent Lie superalgebra when $m+n = 2$, i.e., $L \cong H(0\mid 1)$. Now by Theorem \ref{the4.5}, we have $\otimes^3L \cong A(0\mid 1)$. Therefore the result holds. As $\gamma_k(L)$ is central, $\gamma_k(L)$ and $L$ act trivially on each other and so by using Lemma \ref{p00}, $\gamma_k(L)\otimes L\cong \gamma_k(L)\otimes L^{ab}$ and $(\varphi(L^2) \otimes L) \otimes L \cong (\varphi(L^2)\otimes L^{ab}) \otimes L^{ab}$. Consider the case when $L/\gamma_k(L)$ is abelian. Then $L$ is a nilpotant Lie superalgebra of class two and $\gamma_k(L) =L^2$. Thus $\dim \otimes^3L/\gamma_k(L) = \dim \otimes^3L/L^2 = (m-r\mid n-s)^3$. By using Lemma \ref{prop4.44}, we have
\begin{align*}
		\dim \otimes^3L &\leq \dim \otimes^3 (L/\gamma_k(L)) + \dim ((\gamma_k(L)\otimes L)\otimes L)\\
		& = (m-r \mid n-s)^3+(r\mid s)\times (m-r\mid n-s)^2\\
		&=(m\mid n)\times (m-r \mid n-s)^2\\
		&=(m+ n)(m+n-(r+s))^2.\\
	\end{align*}
Now if $L/\gamma_k(L)$ is non-abelian, then $\dim L/\gamma_k(L) = m+n-\dim \gamma_k(L)$ and $\dim(L/\gamma_k(L))^2 = r+s- \dim \gamma_k(L)$, thus by the induction hypothesis
\begin{eqnarray}\label{eq4.3}
	\dim \otimes^3(L/\gamma_k(L)) \leq ( m+n-\dim \gamma_k(L))(m+n -(r+s))^2.
\end{eqnarray}
 Lemma \ref{prop4.44} and  \ref{eq4.3} imply that
\begin{align*}
	\dim \otimes^3L &\leq \dim \otimes^3 (L/\gamma_k(L)) + \dim ((\gamma_k(L)\otimes L)\otimes L)\\
	& \leq ( m+n-\dim \gamma_k(L)(m+n -(r+s))^2+\dim \gamma_k(L)(m+n -(r+s))^2\\
	&=(m+ n)(m+n-(r+s))^2.\\
\end{align*}
Let $r=1, s=0$. Then from Lemma \ref{th4.4}, we have $L \cong H(k\mid l)\oplus A(m-2k-1\mid n-l)$ for all $k+l \geq 1$. By using Theorem \ref{the4.5}, the equality holds if and only if $L \cong H(1\mid 0)$.
\end{proof}

\textbf{Disclosure statement:} All authors declare that they have no conflicts of interest that are relevant to the content of this article.

\end{document}